\documentclass[12pt]{amsart}
\usepackage{amssymb,amscd}
\usepackage{verbatim}

\usepackage{amsmath,amssymb,graphicx,mathrsfs,bbm}   
\usepackage{enumerate}
\usepackage[colorlinks=true,allcolors = blue]{hyperref} 

\usepackage{tikz}
\usetikzlibrary{matrix}
\usetikzlibrary{patterns,arrows,decorations.pathreplacing}

\usepackage[all]{xy}

\let\frak\mathfrak

\def\>{\relax\ifmmode\mskip.666667\thinmuskip\relax\else\kern.111111em\fi}
\def\<{\relax\ifmmode\mskip-.333333\thinmuskip\relax\else\kern-.0555556em\fi}
\def\vsk#1>{\vskip#1\baselineskip}
\def\vv#1>{\vadjust{\vsk#1>}\ignorespaces}
\def\vvn#1>{\vadjust{\nobreak\vsk#1>\nobreak}\ignorespaces}
 \let\alb\allowbreak

  \let\ssize\scriptstyle
\let\sssize\scriptscriptstyle

\let\Medskip\medskip
\def\medskip{\par\Medskip}
\let\Bigskip\bigskip
\def\bigskip{\par\Bigskip}

\let\Maketitle\maketitle
\def\maketitle{\Maketitle\thispagestyle{empty}\let\maketitle\empty}

\newtheorem{thm}{Theorem}[section]
\newtheorem{cor}[thm]{Corollary}
\newtheorem{lem}[thm]{Lemma}
\newtheorem{prop}[thm]{Proposition}

\newtheorem{ex}[thm]{Example}
\newtheorem{exmp}[thm]{Example}

\theoremstyle{definition}                                  
\numberwithin{equation}{section}

\theoremstyle{definition}
\newtheorem*{rem}{Remark}

\let\mc\mathcal
\let\nc\newcommand

\let\al\alpha

\let\ka\kappa
\let\la\lambda

\let\phi\varphi

\let\der\partial

\let\ox\otimes

\let\geq\geqslant

\let\leq\leqslant

\let\on\operatorname
\let\bi\bibitem
\let\bs\boldsymbol

\def\C{{\mathbb C}}
\def\Z{{\mathbb Z}}
\def\R{{\mathbb R}}

\def\F{{\mathbb F}}   % new Dec 2019

\def\+#1{^{\{#1\}}}
\def\lsym#1{#1\alb\dots\relax#1\alb}
\def\lc{\lsym,}

\def\End{\on{End}}

\def\id{{\on{id}}}

\def\beq{\begin{equation}}
\def\eeq{\end{equation}}
\def\be{\begin{equation*}}
\def\ee{\end{equation*}}

\nc{\bea}{\begin{eqnarray*}}
\nc{\eea}{\end{eqnarray*}}
\nc{\bean}{\begin{eqnarray}}
\nc{\eean}{\end{eqnarray}}
\let\ga\gamma
\let\Ga\Gamma

\nc{\Il}{{\mc I_{\bs\la}}}
\nc{\bla}{{\bs\la}}
\nc{\Fla}{\F_\bla}
\nc{\tfl}{{T^*\Fla}}
\nc{\GL}{{GL_n(\C)}}
\nc{\GLC}{{GL_n(\C)\times\C^*}}

\let\sd s %% \def\sd{\dot s}

\def\ddk_#1{\kk_{#1}\<\>\frac\der{\der\<\>\kk_{#1}}}

\def\bul{\mathbin{\raise.2ex\hbox{$\sssize\bullet$}}}
\def\intt{\mathchoice
{\mathop{\raise.2ex\rlap{$\,\,\ssize\backslash$}{\intop}}\nolimits}
{\mathop{\raise.3ex\rlap{$\,\sssize\backslash$}{\intop}}\nolimits}
{\mathop{\raise.1ex\rlap{$\sssize\>\backslash$}{\intop}}\nolimits}
{\mathop{\rlap{$\sssize\<\>\backslash$}{\intop}}\nolimits}}

\let\kk q %% Q
\let\cc c

\let\Ko K

\def\GZ/{Gelfand-Zetlin}
\def\KZ/{{\slshape KZ\/}}
\def\qKZ/{{\slshape qKZ\/}}
\def\XXX/{{\slshape XXX\/}}

\def\Sym{\on{Sym}}

\nc{\A}{{\mc A}}

\def\Sing{{\on{Sing}}}

\def\slt{{\frak{sl}_2}}

\nc{\hsl}{\widehat{{\frak{sl}_2}}}

\nc{\BC}{{ \mathbb C}}
\nc{\lra}{\longrightarrow}
\nc{\CO}{{\mathcal{O}}}
\nc{\BZ}{{ \mathbb Z}}
\nc{\hfn}{\hat{\frak{n}}}
\nc\Zs{{\Z/p^s\Z}}
\nc\Zo{{\Zs[z]^0}}
\nc\gr{{\on{gr}}}

\nc\fD{{\frak D}}

\begin{document}

\hrule width0pt
\vsk->

\title[Solutions of  \qKZ/ equations modulo an integer]
{Solutions of the $\slt$ \qKZ/ equations modulo an integer}

\author[Evgeny Mukhin and Alexander Varchenko]
{Evgeny Mukhin and Alexander Varchenko}

\maketitle

\begin{center}

{\it $\kern-.4em^\circ\<$Department of Mathematical Sciences,
Indiana University\,--\>Purdue University Indianapolis\kern-.4em\\
402 North Blackford St, Indianapolis, IN 46202-3216, USA\/}

\vsk.5>

{\it $^{\star}$ Department of Mathematics, University
of North Carolina at Chapel Hill\\ Chapel Hill, NC 27599-3250, USA\/}

%\vsk.5>
%{\it $^{ \star}$ Faculty of Mathematics and Mechanics, Lomonosov Moscow State
%University\\ Leninskiye Gory 1, 119991 Moscow GSP-1, Russia\/}

%\vsk.5>
% {\it $^{ \star}$ Moscow Center of Fundamental and Applied Mathematics
%\\ Leninskiye Gory 1, 119991 Moscow GSP-1, Russia\/}

\end{center}

\vsk>
{\leftskip3pc \rightskip\leftskip \parindent0pt \Small
{\it Key words\/}:  \qKZ/ equations; master polynomials, weight functions, difference differentials.

\vsk.6>
{\it 2020 Mathematics Subject Classification\/}: 11D79 (12H25, 32G34, 33C05, 33E30)
\par}
% 11A07 Congruences; primitive roots; residue systems
% 11B65 Binomial coefficients; factorials; q-identities
% 11D79 Congruences in many variables
% 12H25 p-adic differential equations
% 32G34 Moduli and deformations for ordinary differential equations (e.g., Knizhnik--Zamolodchikov equation)
% 33C05 Classical hypergeometric functions, 2F1
% 33C60 Hypergeometric integrals and functions defined by them (E, G, H and I functions)
% 33C70 Other hypergeometric functions and integrals in several variables
% 33C75 Elliptic integrals as hypergeometric functions
% 33E05 Elliptic functions and integrals
% 33E30 Other functions coming from differential, difference and integral equations
% 33E50 Special functions in characteristic p (gamma functions, etc.)

{\let\thefootnote\relax
\footnotetext{\vsk-.8>\noindent
$^\circ\<${\sl E\>-mail}:\enspace emukhin@iupui.edu\>,
supported in part by Simons Foundation grants \rlap{353831, 709444}
\\
$^\star\<${\sl E\>-mail}:\enspace anv@email.unc.edu\>,
supported in part by NSF grant  DMS-1954266}}

\begin{abstract}
We study the  \qKZ/ difference equations with values in the $n$-th tensor power
 of the vector $\slt$ representation $V$, variables $z_1,\dots,z_n$ and integer step $\kappa$.
For any integer $N$ relatively prime to the step $\kappa$,
 we construct a family of polynomials $f_r(z)$ in variables $z_1,\dots,z_n$ with
  values in $V^{\otimes n}$ such that the coordinates of these polynomials with 
  respect to the standard basis of $V^{\otimes n}$ are polynomials with integer 
  coefficients. We show that the polynomials $f_r(z)$ satisfy the \qKZ/ equations modulo $N$.
 
 Polynomials $f_r(z)$ are modulo $N$ analogs of the  hypergeometric 
 solutions of the \qKZ/ equations given in the form of multidimensional Barnes integrals.
\end{abstract}

%{\small\tableofcontents\par}

\setcounter{footnote}{0}
\renewcommand{\thefootnote}{\arabic{footnote}}

\section{Introduction}

The Knizhnik-Zamolodchikov (\KZ/) differential equations are a system of 
linear differential equations, satisfied by conformal
blocks on the sphere in the WZW model of conformal field theory.
%see \cite{KZ}. %\cite{EFK}.
The quantum Knizhnik-Zamolodchikov (\qKZ/) equations are a difference version of the \KZ/ equations which naturally appear in the representation theory of Yangians (rational case) and quantum affine algebras (trigonometric case). 
% introduced in \cite{FR}. $%, ITIJMN, S}.  
%Both the differential and difference equations appear in applications, for example in
% geometry, see \cite{MO}. 
As a rule one considers the \KZ/ and \qKZ/ equations over the field of complex numbers. Then these differential and difference equations are solved in multidimensional hypergeometric integrals.% , see \cite{CF, DJMM, Mat, SV1, R, V1, TV1, TV2, MuV}. 

\iffalse
The fact that certain integrals of closed differential forms over cycles
satisfy a linear  differential or difference equation 
follows by Stokes' theorem from a suitable cohomological relation,
in which the  result of the
application of the corresponding differential or difference 
operator to the integrand of an integral equals the differential of a form of one degree less.
Such cohomological relations for the \KZ/ differential equations were developed in \cite{SV1}
and for the \qKZ/ difference equations in \cite{TV2}.
\fi

\vsk.2>

In \cite{SV} %, V3, V4},  
the \KZ/ differential equations were considered modulo an integer  $N$. It turned out that modulo $N$ the \KZ/ equations have a family of polynomial solutions.
%The construction was based on the fact that
%all cohomological relations described in \cite{SV1} are defined over $\Z$ and can be reduced modulo an integer.
The construction of these solutions was analogous to the construction of the multidimensional hypergeometric
solutions, and these polynomial solutions were called the $N$-hypergeometric
solutions. 

%Properties of the $N$-hypergeometric solutions reflect the properties of the hypergeometric solutions, see for example \cite{V2, RV1, RV2}.

\smallskip

In this paper we consider modulo $N$ the rational $\slt$ \qKZ/ equations  with values in the $n$-th tensor power of the vector representation of $V$ and
with an integer step $\kappa$. The \qKZ/ equations for a function  $f(z_1,\dots,z_n)$ with values in $V^{\otimes n}$ have the form
\bea
f(z_1,\dots, z_a-\ka,\dots, z_n)\,=\,
K_a(z\:;\ka)\,f(z),\qquad a=1,\dots,n,
\eea
where linear operators $K_a(z\:;\ka)$ are given in terms of the rational $\slt$ $R$-matrix, see \eqref{K}. 
The operators $K_a(z\:;\ka)$ commute with the diagonal action of $\slt$,
 and, therefore, it is sufficient to solve the \qKZ/ equations only with values in the space of singular vectors of a given weight.

We fix an integer $N$ relatively prime to $\kappa$  and construct 
$V^{\ox n}$-valued
polynomials in $z_1,\dots,z_n$ such that their coordinates in the standard basis of $V^{\otimes n}$ are polynomials with integer coefficients. Then we show that these $V^{\ox n}$-valued
polynomials satisfy the \qKZ/ equations modulo $N$. 

The idea of the construction is as follows.
The hypergeometric complex $V^{\otimes n}$-valued solutions of the \qKZ/ 
 are given in \cite{TV, MuV}. The solutions with values in the subspace of 
 singular vectors of weight $n-2l$ are written as $l$-dimensional integrals of the form
\beq\label{sol form}
f(z)=\int_{{\mathbbm{i}}\R^{l}} \Phi(t,z) \,   w(t,z)\, W(t,z)dt\,.
\eeq
Here $t=(t_1,\dots,t_l)$,  $\Phi(t,z)$ is a scalar master function given as a ratio of
products of
 Euler gamma functions, $   w(t,z)$ is the vector-valued rational weight function, and $W(t,z)$ is a scalar $\kappa$-periodic
 function with respect to all variables $t,z$;\,
  cf. Section \ref{app 2 sec}. To show that $f(z)$ is a solution, one observes that the difference
\beq\label{difference}
\Phi(t,z_1,\dots, z_a-\ka,\dots, z_n)\,    w(t,z_1,\dots, z_a-\ka,\dots, z_n)\,-\, 
K_a(z\:,\kappa)\,\Phi(t,z) \,   w(t,z)
\eeq
is a discrete differential with respect to variables $t_1,\dots,t_l$, that is,  
 it can be written in the form $\sum_{s=1}^l(   G_{a,s}(t)-   G_{a,s}(t_1,\dots, t_s-\ka,\dots, t_l))$ 
 for appropriate functions $G_{a,s}$.
 
Then using Stokes' theorem one shows that the integral of a discrete difference equals zero, and hence $f(z)$ is a solution of the \qKZ/ equations.
The hypergeometric solution $f(z)$ depends on the choice of the $\ka$-periodic function $W(t,z)$ from a suitable 
finite-dimensional vector 
space. The functions $W(t,z)$ of that space provide the convergence of the integral in \eqref{sol form} and  applicability of Stokes' theorem.

\smallskip

One may consider the \qKZ/ difference equations with values in
a tensor product $\ox_{i=1}^nV_i$ of finite-dimensional $\frak{sl}_k$-modules
and construct the corresponding hypergeometric solutions like in formula
\eqref{sol form}.  The functions $\Phi(t,z)$ and $\vec w(t,z)$ used in
 \eqref{sol form} have meaning in that more general situation, see  \cite{TV}. They depend on the Cartan matrix and the highest weights of the modules.
In our case the highest weight of the vector representation is $\la=1$ and the Cartan matrix is $C=(2)$.

\smallskip

%Choosing a periodic function $W(t,z)$ such that $   G_{s,a}W(t,z)$ has no poles when $\Re t_j\leq |\kappa|$ for all $j,s,a$, and such that the integral over the imaginary space $\R^l$ converges, one gets a solution of the \qKZ/ of the form \eqref{sol form}. This is done for some region of parameters $\ka,z$ and for other parameters one performs analytic continuation.

%The functions $\Phi(t,z)$ and $   w(t,z)$ used in
 %\eqref{sol form} have meaning in a more general situation, see  \cite{TV}. They depend on the Cartan matrix and the highest
% weights of the modules. In our case the highest 
% weight of the vector representation is $\la=1$ and the Cartan matrix is $C=(2)$. 

The main idea
of the construction of polynomial solutions of the \qKZ/ equations modulo $N$,
 is to modify in formulas for $\Phi(t,z)$ and $   w(t,z)$
 the numbers $1$ coming from the highest weight and the numbers $2$ coming from
  the Cartan matrix in such a way, that they do not change modulo $N$ but become multiples of the step $\kappa$.  So, we change in the formulas:
\bea
 1\to 1+Nm=-\kappa k, \qquad  2 \to  2+N m'=\kappa k',
\eea 
where $m,m',k,k'$ are integers. This is possible if and only if $N$ and $\kappa$ are relatively prime. 

After that change, the ratio of products of gamma functions in
the master function $\Phi$ becomes
 a polynomial in $t,z$ due to the fundamental property of the gamma function:
\bea
(-\ka)^k\,\frac{\Gamma\big(\frac{z-k\kappa}{-\kappa}\big)}{\Gamma\big(\frac z{-\ka}\big)}
=z(z-\kappa)\dots (z-(k-1)\kappa)=:[z]_k.
\eea
This modified master function (now called the master polynomial)
% master polynomial $\Phi$ 
is a product of linear factors organized into Pochhammer polynomials of the form $[x]_k$ and $[x]_{k'}$
 with appropriate $x$, see Figure \ref{master pic}. The product $\Phi    w$ also becomes a polynomial. Moreover, each of ${n \choose l}$ coordinates of $\Phi    w$ is a sum of $l!$ terms with each term being
  a product of Pochhammer polynomials of the same kind, see Figure \ref{weight pic}.

Then we check that modulo $N$ the difference \eqref{difference} is still a discrete differential. 

It remains to find a way to integrate over $t$ which eliminates discrete differentials 
of polynomials with integer coefficients. Such a procedure is clearly impossible over complex numbers (since the 
discrete derivative map is surjective), but is well-known modulo an integer $N$.
Namely, let a polynomial $G(t,z)$ be a discrete differential with respect to the $t$-variables,
and let
\bea
G(t,z)=\sum_{r_1,\dots,r_l} a_{r_1,\dots, r_l}(z)[t_1]_{r_1}\dots [t_l]_{r_l}, \qquad
a_{r_1,\dots, r_l}(z)\in \Z[z_1,\dots,z_n].
\eea
Then a coefficient
$a_{r_1,\dots, r_l}(z)$ equals zero modulo $N$ if  $r_i\equiv -1 \pmod N$ for $i=1,\dots,l$.

The assignment to a polynomial $G(t,z)$ 
of a coefficient $a_{r_1,\dots, r_l}(z)$ with $r_i\equiv -1 \pmod N$ for $i=1,\dots,l$,
may be considered as  a modulo $N$ analog of the integral over $t$ of the polynomial $G(t,z)$. We call it a difference $r$-integral. There is a generalization of the difference $r$-integral defined for all sequences of non-negative integers $r=(r_1,\dots,r_l)$ which assigns to a polynomial $G(t,z)$ the polynomial $N_r a_{r_1,\dots, r_l}(z)$, where $N_r$ is the smallest integer such that $N_r(r_i+1)\equiv 0 \pmod N$.
This assignment also gives zero in 
$\Z/NZ$ ,  if $G(t,z)$ is a discrete differential. See Section \ref{int sec}.

Thus, if
$$
\Phi(t,z)    w(t,z)=\sum_{r_1,\dots,r_l}   c_{r_1,\dots, r_l}(z)[t_1]_{r_1}\dots [t_l]_{r_l},
$$
and $(r_1,\dots, r_l)$ is a sequence of non-negative integers
%such that $r_i\equiv -1 \pmod N$ for $i=1,\dots,l$, 
then 
$   f_r=N_r   c_{r_1,\dots, r_l}(z)$ is  a solution  of the \qKZ/ equations modulo $N$. Keeping in mind the language of solutions over complex numbers, we call the solutions constructed in this way, the $N$-hypergeometric solutions.

\smallskip

Since we consider only the tensor products of $\slt$ vector representations, 
and since we have no gamma functions and no  analytic issues, the proofs are simpler than in \cite{MuV} or \cite{TV} and are combinatorial in nature. Moreover, we show that our solutions satisfy the (stronger) symmetric form of the \qKZ/
equations, see \eqref{symm KZ}. This happens due to the symmetric nature of our ``integration''.

\smallskip

The $N$-hypergeometric solution are non-homogeneous polynomials in $z_1,\dots,z_n$. 
One can consider the set of solutions of the form $\sum_r g_r(z)   f_r(z)$, where $g_r(z)$ are scalar polynomials with integer coefficients satisfying
the congruences
$$
 N_r g_r(z_1,\dots, z_i-\ka,\dots,z_n) \equiv
N_r g_r(z_1,\dots, z_i,\dots,z_n) \pmod N 
$$
for all $i$.
It is an interesting problem to describe this set as a module over the ring of  polynomials which are periodic modulo $N$.
 It is nontrivial because this module is not free in general and the answer depends
  on how the integers $\ka, N, n,l$ are related to each other.
We do not address this question in this first joint paper on the subject, but we
 show that the top degree terms of an $N$-hypergeometric solutions of the \qKZ/ 
 equations constitute an $N$-hypergeometric solution of the  \KZ/ equations 
 constructed in \cite{SV}, see Proposition \ref{qkz-kz}, for which some studies have been done, see \cite{V1, V2, V3}.

In contrast to the complex-valued case, the  $N$-hypergeometric solutions is generically do 
not span  the space of singular vectors. Therefore, a natural question is if there are  polynomial modulo $N$ solutions of the \qKZ/ equations which are not $N$-hypergeometric. 

In this paper we consider the \qKZ/ equations with values in a tensor power of the vector
representation of $\slt$. In the same way we  may construct $N$-hypergeometric 
solutions of the \qKZ/ equations with values in a
tensor product of finite-dimensional $\frak{sl}_k$-modules. 

\smallskip

The paper is organized as follows.
In Section \ref{sec 2} we define the \qKZ/ equations.
In Section \ref{diff sec} we discuss the generalities of discrete differentiation and integration modulo $N$.
In Section \ref{master and weight sec} we define the main ingredients of our solutions: master polynomial and weight functions.
In Section \ref{main sec} we formulate and prove our main theorem which provides $N$-hypergeometric solutions of the \qKZ/  equations. These are vectors of non-homogeneous polynomials in $z$ with integer coefficients solving the \qKZ/ equations modulo $N$. In Section \ref{limit sec} we show that the sum of the top degree terms of the $N$-hypergeometric solutions of the \qKZ/ equations 
coincide with the $N$-hypergeomtric solutions of
the \KZ/ equations. In Appendix \ref{sec a} we remind the constructions of the hypergeometric solutions  of the \KZ/ and \qKZ/ equations
 in the special case in which the corresponding hypergeometric integrals are one-dimensional.

\smallskip
The authors thank Bonn University and the Max-Planck Institute for Mathematics in Bonn
for hospitality in May-June of 2022 when this paper was conceived.

\section{Difference \qKZ/ equations}
\label{sec 2}

\subsection{Notations}

Consider the Lie algebra $\slt$ with basis $e,f,h$ and relations
$[e,f]=h$, $[h,e]=2e$,
$ [h,f]=-2f$.
Let $V$ be the two-dimensional $\slt$-module  with basis $v_1,v_2$ and the 
action\ 
$ev_1=0, ev_2=v_1$, $fv_1=v_2, fv_2=0$, $hv_1=v_1, hv_2=-v_2$.

Fix a positive integer $n>1$. The $\slt$-module $V^{\ox n}$ has weight decomposition
\bea
V^{\ox n}=\oplus_{l=0}^n V^{\ox n}[n-2l],
\eea
where $V^{\ox n}[n-2l]$ is the eigenspace of $h$ with eigenvalue $n-2l$.
Let $\Sing V^{\ox n}[n-2l]\subset V^{\ox n}[n-2l]$ be the subspace of
singular vectors  (the vectors annihilated by $e$).

\vsk.2>

Let $\mc I_l$ be the set of all $l$-element subsets of $\{1,\dots,n\}$.
Denote
\bea
v_I=v_{i_1}\ox\dots\ox v_{i_n}\ \in \ V^{\ox n}, 
\eea
where $i_j=2$ if $i_j\in I$ and $i_j=1$ if $i_j\notin I$.
The set $\{v_I\ |\ I\in\mc I_l\}$ is a basis of $V^{\ox n}[n-2l]$.

\subsection{The \qKZ/ equations}
\label{sec qkz}

Define the rational $\,R$-matrix acting on $\:V^{\ox 2}$,
\vvn.2>
\be
R(u)\,=\,\frac{u-P}{u-1}\,,
\vv.3>
\ee
where $P$ is the permutation of factors of $V^{\ox 2}$.
The $R$-matrix satisfies the Yang-Baxter and unitarity equations,
\vvn.4>
\begin{gather}
\label{YB}
R^{(12)}(u-v)R^{(13)}(u)R^{(23)}(v)\,=\,
R^{(23)}(v)R^{(13)}(u)R^{(12)}(u-v)\,,
\\[2pt]
\label{unit}
R^{(12)}(u)R^{(21)}(-u)\,=\,1\,.
\\[-14pt]
\notag
\end{gather}
The first equation is an equation in $\End(V^{\ox 3})$.
The superscript indicates the factors of $\:V^{\ox 3}$
on which the corresponding operators act.

\vsk.2>
Let $\,z=(z_1\lc z_n)$.
Define the \qKZ/ operators $\,K_1,\dots, K_n\>$ acting on $V^{\ox n}$:
\bean
\label{K}
&&K_a(z;\ka) =
R^{(a,a-1)}(z_a-z_{a-1}-\ka)\,\dots\,R^{(a,1)}(z_a-z_1-\ka)
\\
\notag
&&\hspace{120pt}\times\ 
R^{(a,n)}(z_a-z_n)\,\dots\,R^{(a,a+1)}(z_a\<-z_{a+1})\,,
\eean
where $\ka$ is a  parameter.

The \qKZ/ operators preserve the weight decomposition of $V^{\ox n}$,
commute with the $\slt$-action, 
and form a discrete flat connection with step $\ka$,
\bea
K_a(z_1,\dots, z_b-\ka,\dots, z_n;\ka)\,K_b(z;\ka) 
=K_b(z_1,\dots, z_a-\ka\dots, z_n;\ka)  \,K_a(z;\ka)
\eea
for $a,b=1,\dots,n$, see \cite{FR}.

The system of difference equations with step $\ka$,
\vvn.2>
\beq
\label{Ki}
f(z_1,\dots, z_a-\ka,\dots, z_n)\,=\,
K_a(z\:;\ka)\,f(z),\qquad a=1,\dots,n,
\vv.3>
\eeq
for a \,$ V^{\ox n}$-valued
function $f(z)$ is called the \qKZ/ {\it equations\/}.

\vsk.2>
Since the \qKZ/ operators commute with the action of $\slt$ in $V^{\otimes n}$, the \qKZ/ operators preserve the subspace
$\Sing V^{\ox n}[n-2l]$ for any integer $l$. 

\iffalse
A solution $f(z)$ is $\Sing V^{\ox n}[n-2l]$-valued if
\bea
f(z) =\sum_{I\in\mc I_l} f_I(z) \,v_I,
\eea
for suitable scalar functions $f_I(z)$ and $ef(z)=0$.
\fi
\vsk.2>

\subsection{The symmetric \qKZ/ equations}
Let $P^{(a,a+1)}$ be the operator swapping the $a$-th and $a+1$-st tensor factors of $V^{\otimes n}$. 

Let $\mu\in S_n$ be the cyclic permutation $(1,2,\dots,n)$. In terms of simple transpositions we have
\bean
\label{mu}
\mu= s_{1,2} s_{2,3}\dots s_{n-1,n}\,.
\eean
Set 
\bea
P^{(\mu)}=P^{(1,2)}P^{(2,3)}\dots P^{(n-1,n)}.
\eea 
 The following system of equations is called the {\it symmetric \qKZ/ equations}:
 \bean\label{symm KZ}
f(z_1,\dots,z_{a+1}, z_a, \dots,z_n)&=&P^{(a,a+1)} R^{(a,a+1)}(z_a-z_{a+1}) \,f(z), \quad a=1,\dots, n-1,\notag \\
f(z_1-\kappa,z_2,\dots,z_n)&=& P^{(\mu)} f(z_2,\dots,z_n,z_1). 
 \eean
 
The symmetric  \qKZ/ equations imply the  \qKZ/ equations.
\begin{lem}
Let $f(z)$ satisfy \eqref{symm KZ}. Then $f(z)$ satisfies \eqref{Ki}.
\end{lem}
 \begin{proof}
For \eqref{Ki} with $a=1$ we have 
\bea
&&f(z_1-\ka,z_2,\dots, z_n)=P^{(1,2)}\dots P^{(n-1,n)} f(z_2,\dots,z_n,z_1)\\
&& = P^{(1,2)}\dots P^{(n-2,n-1)} P^{(n-1,n)} P^{(n-1,n)}R^{(n-1,n)}(z_1-z_n) f(z_2,\dots,z_{n-1},z_1,z_n)\\
&& =P^{(1,2)}\dots P^{(n-2,n-1)} R^{(n-1,n)}(z_1-z_n) P^{(n-2,n-1)} R^{(n-2,n-1)}(z_1-z_{n-1}) f(z_2,\dots,z_1,z_{n-1},z_n)\\
&&=P^{(1,2)}\dots P^{(n-3,n-2)} R^{(n-2,n)}(z_1-z_n)R^{(n-2,n-1)}(z_1-z_{n-1}) f(z_2,\dots,z_1,z_{n-1},z_n)=\dots\\
&& = R^{(1,n)}(z_1-z_n)\dots R^{(1,2}(z_1-z_2)f(z)=K_1(z\:,\ka) f(z).
\eea
Now we use it to show \eqref{Ki} with $a=2$,
\bea
&&
f(z_1,z_2-\ka,\dots, z_n)=P^{(12)}R^{(12)}(z_2-z_{1}-\ka)f(z_2-\kappa,z_1,\dots,z_n)
\\
&&
=P^{(1,2)} R^{(1,2)}(z_2-z_{1}-\ka)R^{(1,n)}(z_2-z_n)\,\dots\,R^{(1,2)}(z_2-z_{1})f(z_2,z_1,\dots,z_n)
\\
&&
=P^{(1,2)} R^{(1,2)}(z_2-z_{1}-\ka)R^{(1,n)}(z_2-z_n)\,\dots\,R^{(1,2)}(z_2-z_{1})P^{(12)}R^{(1,2)}(z_2-z_1)f(z)
\\
&&
=
\,K_2(z;\ka)\,f(z).
\eea
The proof of equations \eqref{Ki} with $a>2$ is similar.
 \end{proof}
 
In this paper we fix  relatively prime integers $\ka$, $N$. 
We consider equations \eqref{symm KZ} modulo $N$,
 and without loss of generality, we assume that $0<\ka<N$. 
 We construct vectors in $V^{\ox n}[n-2l]$ whose coefficients
  in the basis $v_I$ are polynomials in variables $z$ with integer coefficients,
\beq\label{sol gen form}
   f(z) =\sum_{I\in\mc I_l} f_I(z) \,v_I, \qquad f_I(z)\in \Z[z], 
\eeq
such that equations \eqref{symm KZ} hold modulo $N$ and such that
$e   f(z)=0 $ modulo $N$.

Difference equations modulo $N$ have trivial symmetries. 
Namely, let $   f(z)$ be a polynomial solution modulo $N$ of a difference equation with step $\kappa$. Let $g(z)$ be a periodic polynomial, 
$g(z_1,\dots,z_i-\ka,\dots,z_n)\equiv g(z) \pmod N$ for all $i$. Let $   h(x)$ be any polynomial in variables $z$ with values in $V^{\otimes n}$ 
whose coordinates have integer coefficients.  Then $   f(z) g(z)+N   h(x)$ is also a solution of 
the difference equations modulo $N$.  Moreover, if $N=dd'$, $d,d'\in\Z$, and all 
coefficients of all coordinates $f_I(z)$ of $   f(z)$ are divisible by $d$, then  
$   f(z) g(z)$ is a solution for any  $g(x)$ which is periodic modulo $d'$.
 
\iffalse 
\subsection{More general \qKZ/ equations}

More general \qKZ/ operators depend on an additional parameter $q$, 
\bean
\label{gK}
K_a(z;q;\ka) 
&=&
R^{(a,a-1)}(z_a-z_{a-1}-\ka)\,\dots\,R^{(a,1)}(z_a-z_1-\ka)
\\
\notag
&\times&
q^{h^{(a)}} R^{(a,n)}(z_a-z_n)\,\dots\,R^{(a,a+1)}(z_a\<-z_{a+1})\,.
\eean
The operators $K_a(z;q;\ka) $ form a discrete flat connection with step $\ka$,
 preserve the weight decomposition of $V^{\ox n}$ but do not commute with the $\slt$-action.
In this paper we do not consider the \qKZ/ equations associated with 
the operators $K_a(z;q;\ka)$.

\fi

\section{Difference differentials}\label{diff sec}

Assume again that the integers $\ka, N$ are relatively prime and $0<\ka<N$.
Let $k,k'$ be the integers such that $0<k<N$, $0<k'<N$, 
\bean
\label{C}
\ka k\,  \equiv\, -1, \qquad
\ka k'\,\equiv\, 2 \pmod{N}.
\eean
The integers $k,k'$ exist and are unique. 

\subsection{Pochhammer polynomials}
Let $m$ be a positive integer. Define the Pochhammer polynomial
\bea
[t]_m=\prod_{i=1}^{m}(t-(i-1)\ka).
\eea
We have
\bea
[t-\ka]_m = [t]_m\,\frac{t-\ka m}t,
\qquad
[t+\ka]_m = [t]_m\,\frac{t+\ka}{t-(m-1)\ka}.
\eea
We also call a Pochhammer polynomial $[t]_m$ a {\it string} of length $m$ 
which starts at $t$ and ends at $t-(m-1)\ka$.

The string $[t]_m$ should be thought of as a discretized analog of the monomial $t^m$. For example, in parallel to the monomials, there is a Newton binomial formula for strings as well:
\beq\label{Newton}
[t+z]_m=\sum_{i=1}^m {m\choose i} [t]_i[z]_{m-i}.
\eeq

If $N=p$ is prime, and $\kappa$ is still relatively prime to $p$, then 
 \bea
[t]_p\equiv t^p-t \pmod{p}
\eea
by the Little Fermat theorem. 
In this case we also have  
\bea
[t+z]_p\equiv (t+z)^p-(t+z)\equiv t^p-t+z^p-z \equiv [t]_p+[z]_p.
\eea 

\smallskip 

For $M\in\Z$, we call a polynomial $f(t)\in \Z[t]$ an {\it $M$-constant} if
 $f(t-\kappa)\equiv f(t) \pmod M$. We note that strings, whose length is a multiple of $N$, are $N$-constants. More generally, we have 
\bea
b[t-\ka]_{a}=b[t]_a-ab\kappa [t-\ka]_{a-1}. 
\eea 
Therefore, the polynomial $b[t]_a$ is an $N$-constant if and only if $ab\equiv 0 \pmod N$.

\smallskip 
Let $A$ be a $\Z$-algebra. Our main example is $A=\Z[z_1,\dots,z_n]=\Z[z]$. 
The strings $\{[t]_m,\ m\geq 0\}$ form an $A$-basis of $A[t]$ 
which should be thought of as a discrete analog of the monomial basis $\{t^m, \ \ m\geq 0\}$.

More generally, $\{\prod_{i=1}^l[t_i]_{m_i}, \ m_i\in\Z_{\geq 0}\}$ form an $A$-basis of $A[t_1,\dots,t_l]$ which we will often use.

\subsection{Difference $r$-integrals}
\label{int sec}

Let $r=(r_1,\dots,r_l)$ be a sequence of non-negative integers. 
Let $N_r$ be the least positive integer such that $N_r(r_i+1)\equiv 0 \pmod N$ for all $i$. 
In terms of greatest common divisors (gcd) and least common multiples (lcm) we have
\bea
N_r=\frac{N}{\on{gcd}(r_1+1,\dots,r_l+1,N)}= {\rm {lcm}}\left(\frac{N}{{\rm {gcd}}(N,r_1+1)},\dots, \frac{N}{{\rm {gcd}}(N,r_l+1)}\right).
\eea 
We also set 
\bea
M_r=\frac{N}{N_r}= \on{gcd}(r_1+1,\dots,r_l+1,N).
\eea 

For a polynomial $f(t_1,\dots,t_l)=\sum_{m_1,\dots,m_l} c_{m}\prod_{i=1}^l[t_i]_{m_i}\in A[t_1,\dots,t_l]$ 
and a sequence $r=(r_1,\dots,r_l)$ of  non-negative integers,
we define the {\it difference $r$-integral of $f$}  by the formula:
\bea 
\{f\}^{t_1,\dots,t_l}_r=N_rc_r.
%\quad {\rm if} 
%\quad  
%f(t)=\sum_{m_1,\dots,m_l} c_{m}\prod_{i=1}^l[t_i]_{m_i}.
\eea
The difference $r$-integral is an  $A$-linear map $\{\cdot \}^{t_1,\dots,t_l}_r : A[t_1,\dots,t_l]\to A$. In the next 
section we show that this map vanishes on discrete differentials modulo $N$.

If $r$ is such that $r_i\equiv -1 \pmod N$, $i=1,\dots,l$, then $N_r=1$.
In this case we call the sequence $r$ a {\it maximal sequence } and 
call the   difference $r$-integral a {\it maximal difference integral}.

If $r$ is such that for some $i$ the number $r_i+1$ is relatively prime to $N$,
 then $N_r=N$ and the corresponding  difference $r$-integral is 
 identically zero modulo $N$. We call such $r$ a {\it trivial sequence}
 and  the difference $r$-integral  a {\it trivial difference integral}.

If $N=p$ is prime, then all non-trivial difference integrals are maximal.

\vsk.2>

\iffalse
\subsection{Congruences}
\label{sec congr}

Fix  relatively prime integers $0<\ka<N$.

Let $k,k'$ be the least positive integers such that
\bean
\label{C}
\ka k\,  \equiv\, -1, \qquad
\ka k'\,\equiv\, 2 \pmod{N}.
\eean
Then  we have modulo $N$,
\bean
\label{st}
[x-\ka]_k &\equiv&  [x]_k\,\frac{x+1}{x},
\qquad \qquad
\phantom{aaaa}
[x+\ka]_k = [x]_k\,\frac{x+\ka}{x+1+\ka},
\\
\notag
[x+1-\ka]_{k'} &\equiv&  [x+1]_{k'}\,\frac{x-1}{x+1},
\qquad
[x+1+\ka]_{k'} = [x+1]_{k'}\,\frac{x+1+\ka}{x-1+\ka}.
\eean
\fi

\subsection{Discrete differentials}\label{diff subsec}
Define the {\it discrete $t$-derivative} $D_t:\ A[t]\to A[t]$,
\bea
D_t f(t)=f(t)-f(t-\ka).
\eea
Similar to usual differentiation, the discrete derivative is an $A$ linear map which satisfy the Leibniz rule:
\bea
D_t (f(t)g(t))= (D_t f(t))\, g(t) +f(t-\ka)\,(D_t g(t)).
\eea

We call polynomials in the image of the discrete $t$-derivative the {\it discrete differentials}.
We have
\beq\label{diff on basis}
D_t[t]_m  = m \ka \,[t-\ka]_{m-1},
\eeq
Modulo $N$, the discrete derivative $D_t$ has kernel given by $N$-constants, generated by 
\bea
 a[t]_b, \qquad  a,b\in\Z_{>0},\ ab\equiv 0  \pmod N.
\eea 
For any $N$-constant $g(t)$ we have
\bea
D_t (g(t)\, f(t))\equiv g(t)\, D_t  f(t) \pmod N,
\eea
which explains the name ``$N$-constant".

The next proposition asserts that a difference integral computed on a discrete differential is zero.

\begin{prop}
Let $r=(r_1,\dots,r_l)$, $t=(t_1,\dots,t_l)$, and $f\in A[t_1,\dots,t_l]$. Then for any $j=1,\dots,l$,  
\bea
\{ D_{t_j}f \}_r^{t_1,\dots,t_l}\equiv 0 \pmod{N}.
\eea
\end{prop}
\begin{proof}
Write $f$ in the $A$-basis $\prod_{i=1}^l[t_i]_{m_i}$ and apply $D_{t_j}$ using \eqref{diff on basis}. That gives $D_{t_j}f$ 
in the same basis, and the coefficient in $D_{t_j}f$ of each $[t_j]_{r_j}$ is a multiple of $r_j+1$. 
Therefore, it is zero modulo $N$ when multiplied by $N_r$.
\end{proof}

\subsection{The module of $t$-periods}

In this section we  discuss the set of all difference $r$-integrals of a polynomial $f(t,z)\in\Z[t,z]$.
For simplicity, we consider the case when $t$ and $z$ are single variables.
The results can be extended to the case when $t=(t_1,\dots,t_l)$, $z=(z_1,\dots, z_n)$.
At the end of this section we formulate a remark on  an analog of Fubini's theorem  for the difference $r$-integrals.

\smallskip 

Given a polynomial
$f(t,z)=\sum_{r=0}^m c_r(z)[t]_r\in\Z[t,z]$, we consider the following set of polynomials in 
$z$ with integer coefficients. The set is denoted by $\{f(t,z)\}^t$ and consists of all polynomials of the form
\bea
\sum_{r=1}^m g_r(z)N_rc_r(z) = \sum_{r=1}^m g_r(z)\{f(t,z)\}^t_r\,,
\eea
where each $g_r(z)$ is an arbitrary polynomial such that $N_rg_r(z-\kappa)\equiv N_rg_r(z) \pmod N$, 
in other words, each  $g_r(z)$ is an arbitrary $M_r$-constant.
The set $\{f(t,z)\}^t$  is called the {\it module of $t$-periods of the polynomial $f(t,z)$ relative to the variable $t$ 
and the integer $N$.}  The module is the object of study in this section.
 
 \vsk.2>
The reason for this concept is the following observation.
If a polynomial $f(t,z)$ satisfies a difference equation up to a $t$-derivative:
\bea
f(t,z-\kappa)-A(z)f(t,z)\equiv D_t(h(t,z)) \pmod N,
\eea
 then all polynomials $u(z)\in\{f(t,z)\}^t$ are solutions of the  equation  
 \bea
 u(z-\kappa)\equiv A(z)u(z) \pmod N.
\eea

We have $\{f(t,z)\}^t\subset \Z[z]$, but we are really interested in the projection of  $\{f(t,z)\}^t$ to $\Z/N\Z[z]$. 
The image of  $\{f(t,z)\}^t\}$ in $\Z/N\Z[z]$ is a $\Z/N\Z$-module.

\smallskip 

The definition of the module of $t$-periods uses the $A$-basis  $\{[t]_m,\ m\geq 0\}$. The next proposition says that one can equivalently use the $A$-basis $\{[t+n_1z+n_2]_m,\ m\geq 0\}$ for any fixed
$n_1,n_2\in\Z$ without changing the module of $t$-periods.

\begin{prop}\label{shift and int}

Let $n_1,n_2\in\Z$ and
 \bea
\sum_{r=0}^m c_r(z)[t]_r\equiv \sum_{r=0}^m b_r(z)[t+n_1z+n_2]_r \pmod{N},
\eea 
where $c_r(z),b_r(z)\in\Z[z]$. Then the following two sets are equal: 
\bean\label{two sets}
&&
\left\{\sum_{r=0}^m g_r(z)N_rc_r(z)
\mid \on{each}\  g_r(z)  {\rm{\ is\ an\ arbitrary}} \ M_r\on{-constant}
\right\}
\\
&&
\phantom{aaaa}
\notag
= \left\{\sum_{r=0}^m g_r(z)N_rb_r(z)
\mid \on{each}\  g_r(z) {\rm{\ is\ an\ arbitrary\ }} M_r\on{-constant}
\right\} .
\eean
\end{prop}

\begin{proof}
We give the proof for the case of $n_2=0,n_1=1$. The case $n_2=0$ and non-zero $n_1$ is the same. The case $n_1=0$ is similar (and simpler). The general case if is the combination of the previous two statements.

By \eqref{Newton},
\bea
b_r(z)[t+z]_r=b_r(z)\sum_{i=0}^r {r\choose i} [t]_i[z]_{r-i}.
\eea
We claim that if polynomial $g_i(z)\in\Z[z]$ is an $M_i$-constant, then the polynomial 
\bea
\tilde g_r(z) ={r\choose i}\frac{N_i}{N_r}\, [z]_{r-i} g_i(z)
\eea 
has integer coefficients and, moreover, it is an $M_r$-constant.

It is enough to prove this for the case $N=p^a$
 where $p$ is prime. Let $r+1=p^bx$, $i+1=p^cy$, where $x,y$ are integers not divisible by $p$. 

Let us consider $b\leq a$, $c\leq a$, the other cases are similar. Then $N_r=p^{a-b}$, $N_i=p^{b-c}$. If 
$c\leq b$, then $r-i$ is divisible by $p^c$ and $[z]_{r-i}g_iN_i/N_r$ is an $M_r$-constant since $iN_i/N_r$ is divisible by $p^b$. If $b<c$ then we use ${r \choose i}$. It is well known that a binomial coefficient is divisible by $p^d$ where $d$ is the number of moving a digit when the difference $r-i$ is computed in the base $p$ system. It follows that ${r \choose i}$ is divisible by $p^{c-b}$ and the claim follows.

Thus  $g_i(z)N_ic_i(z)=g_i(z)N_ib_i(z)+\sum_{r=i+1}^n b_r(z) N_r \tilde g_r$ where $\tilde g_r$ are $M_r$-constants. It follows that the 
left-hand side set in \eqref{two sets} is inside of the right set. 
Doing the change  $t\to t-z$ and repeating the argument, we obtain the other inclusion.
\end{proof}

Thus, a difference $r$-integral of $f(t,z)$ with respect to $t$ may change, if the $A$-basis is shifted by a linear polynomial in 
$z$,  but the module of $t$-periods does not.

Another statement in the same spirit is the following proposition.

\begin{prop}
\label{per and int}

For any integers $n_1,n_2,n_3$, with $n_3>0$, and any polynomial $f(t,z)\in\Z[t,z]$, we have
\bea
\{f(t,z)[t+n_1 z+n_2]_{n_3N} \}^t\equiv \{f(t,z) \}^t \pmod{N}.
\eea 
\end{prop}
\begin{proof}
The proof is similar to Proposition \ref{shift and int}. We treat the case $n_1=n_3=1$ and $n_2=0$. 

We have modulo $N$
\bea
[t]_r[t+z]_N \equiv [t]_r[(t-r\kappa)+z]_N=[t]_r\sum_{i=0}^N {N\choose i} [t-r\kappa]_{i}[z]_{N-i} 
=\sum_{i=0}^N {N\choose i}[t]_{r+i}[z]_{N-i}.
\eea 
The claim is that if $g_r(z)$ is an $M_r$-constant, then
\bea
\tilde g(z)={N \choose i}\frac{N_{r+i}}{N_r}[z]_{N-i}g_r(x)
\eea 
is an $M_{r+i}$-constant. This is checked for $N=p^a$ where $p$ is prime, by using 
the fact that if $i=p^cx$, where $x$ is an integer not divisible by $p$ and $c<a$, then ${N \choose i}$ is divisible by $p^{a-c}$.

Write $f(t,z)=\sum_{r=0}^m c_r(z)[t]_m$ and  $f(t,z)[t+z]_N=\sum_{r=0}^{N+n} b_r(z)[t]_r$. 
It follows from the claim that  for any $M_i$-constant $g_i(z)$, we have
$b_i(z)g_i(z)N_i=\sum_{r=i-N}^i c_r(z) N_r \tilde g_r(z)$
 where we set $c_r(z)=0$ for $r<0$ and where $\tilde g_r(z)$ are $M_r$-constants. Therefore,
 $\{f(t,z)[t+z]_N\}^t\subset \{f(t,z)\}^t$ modulo $N$.

Clearly $N_i=N_{i+N}$ and $g_i(z)$ is an $M_i$-constant if and only if $g_i(z)$ is an $M_{i+N}$-constant. 
Therefore, if $i\leq m$, then $b_{i+N}(z)g_i(z)N_{i+N}=c_i(z)N_ig_i(z)+\sum_{r=i+1}^{i+N} c_r(z) N_r \tilde g_r(z)$. 
Thus we obtain $c_i(z)N_ig_i(z)\in \{f(t,z)[t+z]_N\}^t$ recursively starting from $i=m$, then continuing to $i=m-1$, then to $m-2$,
 and so on. 
\end{proof}

We  use   Propositions  \ref{shift and int} and \ref{per and int}
 in Section \ref{main sec} to explain that the family of
 the modulo $N$  solutions of the \qKZ/ equations constructed in Theorem \ref{main thm}
 does not depend on various choices.

\begin{rem}
Consider an example.
Let $f(t_1,t_2,z)=p(z)[t_1]_2[t_2]_2$ and $N=4$. Then 
\bea
\{\{f\}_2^{t_1}\}^{t_2}_2=\{\{f\}_2^{t_2}\}^{t_1}_2=4p(z) \equiv 0\pmod{4}, \quad {\textrm {but}} \quad \{f\}_{2,2}^{t_1,t_2}=2p(z). 
\eea 
More generally, we have an  $N$-analog of Fubini's theorem:
\bean\label{Fubini}
&&
\{ \{ f(t_1,t_2,z)\}^{t_1}_{r_1}\}^{t_2}_{r_2}=
\{ \{ f(t_1,t_2,z) \}^{t_2}_{r_2}\}_{r_1}^{t_1}
\\
\notag
&&
\phantom{aaaaaa}
=
\frac{N_{r_1}N_{r_2}}{\on{gcd}(N_{r_1},N_{r_2})} \{f(t_1,t_2,z)\}^{t_1,t_2}_{r_1,r_2}
=\frac{N_{r_1}N_{r_2}}{\on{gcd}(N_{r_1},N_{r_2})}\{f(t_1,t_2,z)\}^{t_2,t_1}_{r_2,r_1}.
\eean

Note that the integer $N_{r_1,r_2}$ is a divisor of $N_{r_1}N_{r_2}$ and, in general, 
 is a proper divisor. Thus, in general, the double integration  gives more  difference $r$-integrals than the repeated integration.

\end{rem}

\section{Master function and weight functions}
\label{master and weight sec}

We define the main ingredients of
 the construction of the solutions of the \qKZ/ equation modulo $N$: the master polynomial and weight function.

\subsection{Master polynomial}
Recall integers $N,\kappa, k, k'$ defined in \eqref{C}.
We also have integers $n,l$ corresponding to the choice of the space $\Sing V^{\ox n}[n-2l]$.

Let again $z=(z_1,\dots,z_n), t=(t_1,\dots,t_l)$.
Define the master polynomial 
\bea
\Phi(t,z) = \prod_{a=1}^n\prod_{i=1}^l [t_i-z_a]_k \prod_{1\leq i<j\leq l} [t_i-t_j +1]_{k'}\,.
\eea
The zeroes of the master polynomial are shown on Figure \ref{master pic} for  $l=2$. 
The start of each string is indicated in blue. The string itself consist of $k$ or $k'$ zeroes of the master polynomial. The distance
 between neighboring zeroes is $-\kappa$ as shown in red in Figure \ref{master pic}.

Note that the master polynomial is symmetric in variables $z_1,\dots,z_n$. 

As explained in the introduction, the master polynomial is obtained from the master function used in 
 the case of complex-valued solutions. This master function originally is a ratio of products of gamma 
 functions, see Section \ref{app 2 sec}, which reduces to a polynomial under our assumptions.

\iffalse
We have
\bea
\Phi(t_1,\dots,t_i-\ka,\dots,t_r,z)
&=&
 \Phi(t,z)\prod_{a=1}^n \frac{t_i-z_a-\ka k}{t_i-z_a}
\\
&\times &
\prod_{j<i}\frac{t_j-t_i+1+\ka}{t_j-t_i +1-\ka k'+\ka}
\prod_{i<j}\frac{t_i-t_j+1-\ka k'}{t_i-t_j+1}\,,
\\
\Phi(t, z_1,\dots,z_a-\ka,\dots,z_n)
&=&
 \Phi(t,z) \prod_{i=1}^r \frac{t_i-z_a+\ka}{t_i-z_a -\ka k+\ka}\,.
\eea
\fi
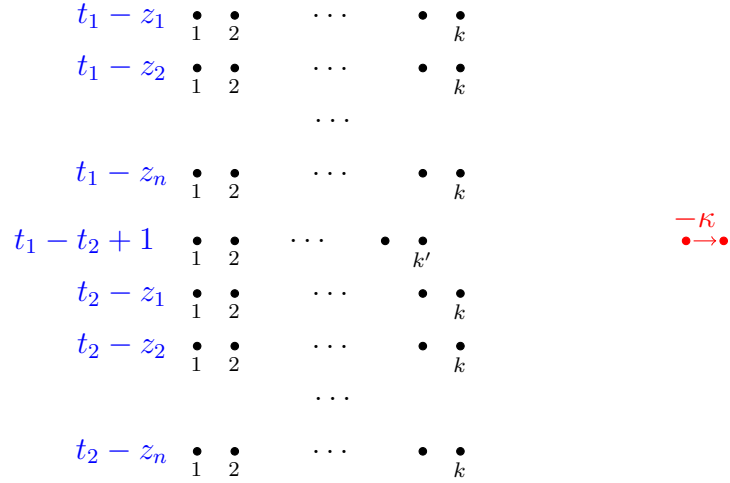
\begin{figure}
\begin{tikzpicture}%[scale=1.2]
%\node[blue, thick] at (-4, 4.3) {$\Psi$};
\draw[fill, red ] (1,0.7) circle [radius=0.05];
\draw[fill, red ] (1.5,0.7) circle [radius=0.05];
\draw[red, ->] (1.1,0.7) -- node[above] {$-\ka\ \ $} (1.4,0.7) ; 
\draw[fill] (-2,3.7) circle [radius=0.05] node[below] {${}_{k}$};
\draw[fill] (-2.5,3.7) circle [radius=0.05];
\node at (-3.7,3.7) {$\dots$};
\draw[fill] (-5,3.7) circle [radius=0.05] node[below] {${}_{2}$};
\draw[fill] (-5.5,3.7) circle [radius=0.05] node[below] {${}_{1}$};
\node[blue] at (-6.5,3.7) {$t_1-z_1$};

\draw[fill] (-2,3) circle [radius=0.05] node[below] {${}_{k}$};
\draw[fill] (-2.5,3) circle [radius=0.05];
\node at (-3.7,3) {$\dots$};
\draw[fill] (-5,3) circle [radius=0.05] node[below] {${}_{2}$};
\draw[fill] (-5.5,3) circle [radius=0.05] node[below] {${}_{1}$};
\node[blue] at (-6.5,3) {$t_1-z_2$};
%draw [thick,decorate,decoration={brace,amplitude=10pt,mirror},xshift=0.4pt,yshift=-0.4pt](-5.5,-0.2) -- (-2,-0.2) node[black,midway,yshift=-0.6cm] {\footnotesize $k$};

\node at (-3.7,2.3) {$\Large{\dots}$};

\draw[fill] (-2,1.6) circle [radius=0.05] node[below] {${}_{k}$};
\draw[fill] (-2.5,1.6) circle [radius=0.05];
\node at (-3.7,1.6) {$\dots$};
\draw[fill] (-5,1.6) circle [radius=0.05] node[below] {${}_{2}$};
\draw[fill] (-5.5,1.6) circle [radius=0.05] node[below] {${}_{1}$};
\node[blue] at (-6.5, 1.6) {$t_1-z_n$};

\draw[fill] (-2.5,0.7) circle [radius=0.05] node[below] {${}_{k'}$};
\draw[fill] (-3,0.7) circle [radius=0.05];
\node at (-4,0.7) {$\dots$};
\draw[fill] (-5,0.7) circle [radius=0.05] node[below] {${}_{2}$};
\draw[fill] (-5.5,0.7) circle [radius=0.05] node[below] {${}_{1}$};
\node[blue] at (-7,0.7) {$t_1-t_2+1$};

\draw[fill] (-2,0) circle [radius=0.05] node[below] {${}_{k}$};
\draw[fill] (-2.5,0) circle [radius=0.05];
\node at (-3.7,0) {$\dots$};
\draw[fill] (-5,0) circle [radius=0.05] node[below] {${}_{2}$};
\draw[fill] (-5.5,0) circle [radius=0.05] node[below] {${}_{1}$};
\node[blue] at (-6.5,0) {$t_2-z_1$};

\draw[fill] (-2,-0.7) circle [radius=0.05] node[below] {${}_{k}$};
\draw[fill] (-2.5,-0.7) circle [radius=0.05];
\node at (-3.7,-0.7) {$\dots$};
\draw[fill] (-5,-0.7) circle [radius=0.05] node[below] {${}_{2}$};
\draw[fill] (-5.5,-0.7) circle [radius=0.05] node[below] {${}_{1}$};
\node[blue] at (-6.5,-0.7) {$t_2-z_2$};

\node at (-3.7,-1.4) {$\Large{\dots}$};

\draw[fill] (-2,-2.1) circle [radius=0.05] node[below] {${}_{k}$};
\draw[fill] (-2.5,-2.1) circle [radius=0.05];
\node at (-3.7,-2.1) {$\dots$};
\draw[fill] (-5,-2.1) circle [radius=0.05] node[below] {${}_{2}$};
\draw[fill] (-5.5,-2.1) circle [radius=0.05] node[below] {${}_{1}$};
\node[blue] at (-6.5,-2.1) {$t_2-z_n$};

\end{tikzpicture}\caption{Zeroes of the master polynomial $\Phi(t_1,t_2,z_1,\dots,z_n)$.}\label{master pic}
\end{figure}

\subsection{Symmetrization modulo $N$}
In the case of complex-valued solutions, one uses an action of a symmetric group $S_l$ on functions of $t_1,\dots,t_l$ given by the formula
\bea
(\tau\circ f)(t_1,\dots,t_l) =
f(t_{\tau(1)} ,\dots, t_{\tau(l)}) \prod_{i<j,\,\tau(i)>\tau(j)}  \frac{t_i-t_{j}-1}{t_i-t_{j}+1},
\eea
for any $\tau\in S_l$. We modify it as follows. 

For $\tau\in S_l$  define
\bea
(\tau f)(t_1,\dots,t_l) =
f(t_{\tau(1)} ,\dots, t_{\tau(l)}) \prod_{i<j,\,\tau(i)>\tau(j)}  \frac{t_i-t_{j}+1-\ka k'}{t_i-t_{j}+1}.
\eea
Note that this formula defines an action of $S_l$ modulo $N$ since $1-\ka k'\equiv -1 \pmod N$.

We also define symmetrization modulo $N$:
\bea
\Sym_t f(t) = \sum_{\tau\in S_l} (\tau f)(t)\,.
\eea

\subsection{Weight function.}
In the case of complex-valued solutions, the coordinates of the weight function are defined as symmetrizations of
the functions
\bea
\prod_{i=1}^r \frac {1}{t_i-z_{a_i}} \prod_{l=1}^{a_i-1}\frac {t_i-z_{l}+1}{t_i-z_{l}}.
\eea
We modify it as follows.

For $I=\{1\leq a_1<\dots<a_l\leq n\}\in \mc I_l$,
define the coordinates of the weight function
\bea
w_I(t,z) &=& \Sym_t\left(
\prod_{i=1}^l \frac {1}{t_i-z_{a_i}} \prod_{j=1}^{a_i-1}\frac {t_i-z_{j}-\ka k}{t_i-z_{j}}\right).
\eea
Since $k\ka\equiv -1 \pmod N$,  and $\Sym_t$ coincides modulo $N$ with the previous symmetrization, this formula coincides with the formula used in the complex-valued case modulo $N$.

It is convenient to denote
\bea
w_I^\tau(t,z)=\tau \left(\prod_{i=1}^l \frac {1}{t_i-z_{a_i}} \prod_{j=1}^{a_i-1}\frac {t_i-z_{j}-\ka k}{t_i-z_{j}}\right).
\eea
Then  $w_I(t,z)=\sum_{\tau\in S_l}w_I^\tau(t,z)$.

We also denote
\bea 
U_I(t,z)=\Phi(t,z) w_I(t,z), \qquad U_I^\tau(t,z)=\Phi(t,z) w_I^\tau (t,z).
\eea

Then we have a lemma saying that the summands $U_I^\tau(t,z)$ are polynomials organized in products of strings.

\begin{lem}
For $\tau\in S_l$, the function $U_I^\tau(t,z)$
 is a polynomial which is a product of $\binom{l}{2}$ strings of length $k'$,
of $l(n-1)$ strings of length $k$,  and of $l$ strings of length $k-1$,

\bean
\label{pr}
&&
U_I^\tau(t,z) = \prod_{ i<j, \,\tau(i)<\tau(j)} [t_i-t_j +1]_{k'}
\prod_{ i<j, \,\tau(i)>\tau(j)} [t_i-t_j +1-\ka]_{k'}
\\
\notag
&&
\times \prod_{i=1}^l \Big(\prod_{s=1}^{a_i-1} [t_{\tau(i)}-z_s-\ka]_k \Big) [t_{\tau(i)} -z_{a_i}-\ka]_{k-1}
\Big(\prod_{s=a_i+1}^n [t_{\tau(i)}-z_s]_k \Big) .
\eean
%The strings $\{ [t_{\tau(i)} - z_{a_i}-\ka]_{k-1}\ | \ i=1,\dots,r\}$ of length $k-1$ completely determine all strings of $(\tau U_I)(t,z)$.
\end{lem}
\begin{proof}
It is easy to see that the poles of $w_I^\tau$ cancel with the starts of the
 strings existing in the master polynomial, while the zeroes append those strings.
 
Examples of $U_I^\tau(t,z)$ are shown on Figure \ref{weight pic}. In that picture $l=2$ and $I=\{1,2\}$. Zeroes 
of $U_I^{\id}(t,z)$ are pictured on the left and zeroes of $U_I^{(1,2)}(t,z)$ on the right. The cancelled zeroes of the master polynomial are pictured as red circles. The new zeroes are shown in red bullets.
\end{proof}

\begin{figure}
\begin{tikzpicture}%[scale=1.2]
%\node[blue, thick] at (-4, 4.3) {$\Psi$};

\draw[fill, red] (7,3.7) circle [radius=0.05] node[below] {${}_{k+1}$};
\draw[fill] (6.5,3.7) circle [radius=0.05] node[below] {${}_{k}$};
\draw[fill] (6,3.7) circle [radius=0.05];
\node at (4.8,3.7) {$\dots$};
\draw[fill] (3.5,3.7) circle [radius=0.05] node[below] {${}_{2}$};
\draw[red] (3,3.7) circle [radius=0.05] node[below] {${}_{1}$};
\node[blue] at (2,3.7) {$t_1-z_1$};

\draw[fill] (6.5,3) circle [radius=0.05] node[below] {${}_{k}$};
\draw[fill] (6,3) circle [radius=0.05];
\node at (4.8,3) {$\dots$};
\draw[fill] (3.5,3) circle [radius=0.05] node[below] {${}_{2}$};
\draw (3,3) circle [radius=0.05] node[below] {${}_{1}$};
\node[blue]  at (2,3) {$t_1-z_2$};
%draw [thick,decorate,decoration={brace,amplitude=10pt,mirror},xshift=0.4pt,yshift=-0.4pt](-5.5,-0.2) -- (-2,-0.2) node[black,midway,yshift=-0.6cm] {\footnotesize $k$};

\node[thick] at (4.8,2.3) {$\dots$};

\draw[fill] (6.5,1.6) circle [radius=0.05] node[below] {${}_{k}$};
\draw[fill] (6,1.6) circle [radius=0.05];
\node at (4.8,1.6) {$\dots$};
\draw[fill] (3.5,1.6) circle [radius=0.05] node[below] {${}_{2}$};
\draw[fill] (3,1.6) circle [radius=0.05] node[below] {${}_{1}$};
\node[blue]  at (2, 1.6) {$t_1-z_n$};

\draw[fill, red] (6.5,0.7) circle [radius=0.05] node[below] {${}_{k'+1}$};
\draw[fill] (6,0.7) circle [radius=0.05] node[below] {${}_{k'}$};
\draw[fill] (5.5,0.7) circle [radius=0.05];
\node at (4.5,0.7) {$\dots$};
\draw[fill] (3.5,0.7) circle [radius=0.05] node[below] {${}_{2}$};
\draw[red] (3,0.7) circle [radius=0.05] node[below] {${}_{1}$};
\node[blue]  at (1.5,0.7) {$t_1-t_2+1$};

\draw[fill] (6.5,0) circle [radius=0.05] node[below] {${}_{k}$};
\draw[fill] (6,0) circle [radius=0.05];
\node at (4.8,0) {$\dots$};
\draw[fill] (3.5,0) circle [radius=0.05] node[below] {${}_{2}$};
\draw[red] (3,0) circle [radius=0.05] node[below] {${}_{1}$};
\node[blue]  at (2,0) {$t_2-z_1$};

\draw[fill] (6.5,-0.7) circle [radius=0.05] node[below] {${}_{k}$};
\draw[fill] (6,-0.7) circle [radius=0.05];
\node at (4.8,-0.7) {$\dots$};
\draw[fill] (3.5,-0.7) circle [radius=0.05] node[below] {${}_{2}$};
\draw[fill] (3,-0.7) circle [radius=0.05] node[below] {${}_{1}$};
\node[blue]  at (2,-0.7) {$t_2-z_2$};

\node[thick] at (4.8,-1.4) {$\dots$};

\draw[fill] (6.5,-2.1) circle [radius=0.05] node[below] {${}_{k}$};
\draw[fill] (6,-2.1) circle [radius=0.05];
\node at (4.8,-2.1) {$\dots$};
\draw[fill] (3.5,-2.1) circle [radius=0.05] node[below] {${}_{2}$};
\draw[fill] (3,-2.1) circle [radius=0.05] node[below] {${}_{1}$};
\node[blue]  at (2,-2.1) {$t_2-z_n$};

\draw[fill] (-2,3.7) circle [radius=0.05] node[below] {${}_{k}$};
\draw[fill] (-2.5,3.7) circle [radius=0.05];
\node at (-3.7,3.7) {$\dots$};
\draw[fill] (-5,3.7) circle [radius=0.05] node[below] {${}_{2}$};
\draw[red] (-5.5,3.7) circle [radius=0.05] node[below] {${}_{1}$};
\node[blue]  at (-6.5,3.7) {$t_1-z_1$};

\draw[fill] (-2,3) circle [radius=0.05] node[below] {${}_{k}$};
\draw[fill] (-2.5,3) circle [radius=0.05];
\node at (-3.7,3) {$\dots$};
\draw[fill] (-5,3) circle [radius=0.05] node[below] {${}_{2}$};
\draw[fill] (-5.5,3) circle [radius=0.05] node[below] {${}_{1}$};
\node[blue]  at (-6.5,3) {$t_1-z_2$};
%draw [thick,decorate,decoration={brace,amplitude=10pt,mirror},xshift=0.4pt,yshift=-0.4pt](-5.5,-0.2) -- (-2,-0.2) node[black,midway,yshift=-0.6cm] {\footnotesize $k$};

\node[thick] at (-3.7,2.3) {$\Large{\dots}$};

\draw[fill] (-2,1.6) circle [radius=0.05] node[below] {${}_{k}$};
\draw[fill] (-2.5,1.6) circle [radius=0.05];
\node at (-3.7,1.6) {$\dots$};
\draw[fill] (-5,1.6) circle [radius=0.05] node[below] {${}_{2}$};
\draw[fill] (-5.5,1.6) circle [radius=0.05] node[below] {${}_{1}$};
\node[blue]  at (-6.5, 1.6) {$t_1-z_n$};

\draw[fill] (-2.5,0.7) circle [radius=0.05] node[below] {${}_{k'}$};
\draw[fill] (-3,0.7) circle [radius=0.05];
\node at (-4,0.7) {$\dots$};
\draw[fill] (-5,0.7) circle [radius=0.05] node[below] {${}_{2}$};
\draw[fill] (-5.5,0.7) circle [radius=0.05] node[below] {${}_{1}$};
\node[blue]  at (-7,0.7) {$t_1-t_2+1$};

\draw[fill,red] (-1.5,0) circle [radius=0.05] node[below] {${}_{k+1}$};
\draw[fill] (-2,0) circle [radius=0.05] node[below] {${}_{k}$};
\draw[fill] (-2.5,0) circle [radius=0.05];
\node at (-3.7,0) {$\dots$};
\draw[fill] (-5,0) circle [radius=0.05] node[below] {${}_{2}$};
\draw[red] (-5.5,0) circle [radius=0.05] node[below] {${}_{1}$};
\node[blue]  at (-6.5,0) {$t_2-z_1$};

\draw[fill] (-2,-0.7) circle [radius=0.05] node[below] {${}_{k}$};
\draw[fill] (-2.5,-0.7) circle [radius=0.05];
\node at (-3.7,-0.7) {$\dots$};
\draw[fill] (-5,-0.7) circle [radius=0.05] node[below] {${}_{2}$};
\draw[red] (-5.5,-0.7) circle [radius=0.05] node[below] {${}_{1}$};
\node[blue]  at (-6.5,-0.7) {$t_2-z_2$};

\node at (-3.7,-1.4) {$\Large{\dots}$};

\draw[fill] (-2,-2.1) circle [radius=0.05] node[below] {${}_{k}$};
\draw[fill] (-2.5,-2.1) circle [radius=0.05];
\node at (-3.7,-2.1) {$\dots$};
\draw[fill] (-5,-2.1) circle [radius=0.05] node[below] {${}_{2}$};
\draw[fill] (-5.5,-2.1) circle [radius=0.05] node[below] {${}_{1}$};
\node[blue]  at (-6.5,-2.1) {$t_2-z_n$};

\end{tikzpicture}\caption{Zeroes of polynomials $U_{1,2}^{\id}(t,z)$ and $U_{1,2}^{(1,2)}(t,z)$.}\label{weight pic}
\end{figure}

The polynomials $U_I(t,z)$ are coordinates of the integrand which will be used to construct solutions of the $qKZ$ equations. Now we pass to the integrand itself. 

Define the polynomial $   U(t,z)$ with values in $V^{\otimes n}[n-2l]$
by the formula
\beq\label{U}
   U(t,z) = \sum_{I\in\mc I_l} \,U_I(t,z)\,v_I=\sum_{I\in\mc I_l} \sum_{\tau\in S_l} \,\Phi(t,z)w_I^\tau(t,z)\,v_I.
\eeq
In the next section we show that any difference $r$-integral of 
$U(t,z)$ solves the symmetric \qKZ/ equations modulo $N$ and is a singular vector modulo $N$.

\iffalse
\begin{exmp}  Let $n=2, r=1$. Then
\bea
&&
W_{\{1\}}(t,z) = [t_1-z_1 -\ka]_{k-1} [t_1-z_2]_{k},
\\
&&
W_{\{2\}}(t,z) = [t_1-z_1 -\ka]_{k} [t_1-z_2 -\ka]_{k-1}.
\eea
Let $n=2, r=2$. Then
\bea
W_{\{1,2\}}(t,z) 
&=& 
[t_1-t_2+1]_{k'}
[t_1-z_1 -\ka]_{k-1} [t_1-z_2]_{k}
\\
&\times&
 [t_2-z_1 -\ka]_{k} [t_2-z_2 -\ka]_{k-1}
\\
&+&
[t_1-t_2+1-\ka]_{k'}
[t_2-z_1 -\ka]_{k-1} [t_2-z_2]_{k}
\\
&\times&
 [t_1-z_1 -\ka]_{k} [t_1-z_2 -\ka]_{k-1}.
 \eea

\end{exmp}
\fi

\section{The main theorem}\label{main sec}
In this section we construct solutions modulo $N$ of the \qKZ/ equations.

\subsection{The statement of the theorem.}

\begin{thm}\label{main thm} 
Let $   U(t,z)$ be the polynomial in $t_1,\dots,t_l, z_1,\dots,z_n$ with values in
 $V^{\otimes n}[n-2l]$ defined in \eqref{U}. Then for any $r=(r_1,\dots,r_l)$ the difference $r$-integral $  {f}_r(z)=\{   U(t,z) \}^t_r$ is a polynomial in $z_1,\dots,z_n$ which is a solution of the symmetric \qKZ/ equations \eqref{symm KZ} modulo $N$.
The vector ${f}_r(z)$ is singular modulo $N$.
\end{thm}

We prove the theorem in the next three sections.

\smallskip 

If $g(z)$ is a scalar $N$-constant polynomial and $  {f}(z)$ is a solution of
the \qKZ/ modulo $N$, then $g(z)   {f}(z)$ is a solution too. 
Moreover, we have  a stronger statement.
 Namely, let 
 $f(z) =\sum_{I\in\mc I_l} f_I(z) \,v_I$, \,$f_I(z)\in \Z[z]$, 
 be a solution  of the \qKZ/ equations modulo $N$, 
 %of the form \eqref{sol gen form}
  and let $a\in\Z$ be such that all polynomials  $f_I(z)/a$ have integer coefficients. Then for 
  any scalar $(N/\on{gcd}(a,N))$-constant $g(z)$, the product $g(z)  {f}(z)$ is also a solution of the
  \qKZ/ equations modulo $N$.

For example, if $g_r(z)\in\Z[z]$ is an $M_r$-constant then $g_r(z)  {f}_r(z)$ is also a solution.
Thus, by Theorem \ref{main thm} all functions in the module of $t$-periods $\{U(t,z)\}^t$ are 
 solutions modulo $N$ of the \qKZ/ equations. We call all such solutions {\it $N$-hypergeometric solutions}. 
 Thus, the $N$-hypergeometric solutions are polynomials 
 of the form $\sum_r g_r(z)   {f}_r(z)$ where $  {f}_r(z)\in N_r\Z[z]$ are given 
 in Theorem \ref{main thm}, and  $g_r(z)\in\Z[z]$ are $M_r$-constants. 
 We give examples of $N$-hypergeometric solutions in Section \ref{example sec}. 

\smallskip 

Seemingly, there are several other ways to obtain solutions  modulo $N$.  
The first way is to choose integers $k$ and $k'$ as any positive integer 
solutions of congruences \eqref{C} and not as the least positive integer solutions.
Such a choice will change the master and weight functions.
It is clear that Theorem \ref{main thm} still holds for the modified integrand $U(t,z)$.
However the choice of new $k$ and $k'$ does not produce new solutions by  Proposition \ref{per and int}. 

The second way is to construct solutions by integrating with respect to a different $A$-basis, 
e.g. $\{\prod [t_i-z_i]_{m_i}, m_i\in\Z_{\geq 0}\}$. Then an analog  of
 Theorem \ref{main thm} still holds. However, a different choice of an $A$-basis produces no new solutions as well
 by Proposition   \ref{shift and int}.

The third way is to consider a divisor $\bar N$ of $N$ and a solution $\bar f(z)$  
 of the \qKZ/ equations modulo $\bar N$.  Then $N\bar f(z)/\bar N$ is a solution of the \qKZ/ 
equations modulo $N$. The solutions obtained in this way also are not new.

\subsection{The first $n-1$ equations.}
In this section we check the first $n-1$ equations of \eqref{symm KZ}.
 Moreover, we show that these equations are already satisfied before taking the difference $r$-integral.

Namely, we check modulo $N$
\beq\label{eq}
  {g}(z_1,\dots,z_{a+1}, z_a, \dots,z_n) \equiv P^{(a,a+1)} R^{(a,a+1)}(z_a-z_{a+1}) \,  {g}(z)
\eeq 
for $  {g}(z)=   w(t,z)=\sum_{I\in\mc I_l}w_Iv_I$.

This is sufficient because  multiplication by the scalar master polynomial
 $\Phi(z,t)$ preserves relation \eqref{eq} as $\Phi$ is symmetric in $z_1,\dots, z_n$. 
 The integration preserves this relation as well since \eqref{eq} does not depend on integration variables $t_1,\dots,t_l$.

We have three cases. 

First, let $I$ be such that $a,a+1\not\in I$. Then $w_I$ is symmetric in $z_a,z_{a+1}$.
The vector $v_I$ has $v_1\otimes v_1$ in positions $a,a+1$. For any $x$, we
have $P\, R(x) v_1\otimes v_1=v_1\otimes v_1$. Therefore, $  {g}=w_Iv_I$ satisfies \eqref{eq} for each $I\in\mc I_l$.

Second, let $I$ be such that $a,a+1\in I$. The vector $v_I$ has $v_2\otimes v_2$ in 
positions $a,a+1$. Again, for any $x$, we have $P\, R(x) v_2\otimes v_2=v_2\otimes v_2$. 
So we have to check that $w_I$ does not change when $a$ and $a+1$ are swapped.
 We claim that for any $\tau\in S_l$, the sum $w_I^\tau+w_I^{(a,a+1)\tau}$ has
  this property. Indeed, if $\tau(a)=b$, $\tau(a+1)=c$, and if $c>b$, then modulo $N$ we have
\bea
&&(w_I^\tau+w_I^{(a,a+1)\tau})(t,z)\\
&&=\frac{h(z,t)}{(t_b-z_a)(t_c-z_a)}\left(\frac {t_c-z_{a}-\ka k}{t_c-z_{a+1}}
+ \frac {{t_b-z_{a}-\ka k}}{t_b-z_{a+1}}
\,\frac{t_b-t_c+1-\ka k'}{t_b-t_c+1}\right)\\
&&\equiv 
\frac{h(z,t)}{(t_b-z_{a+1})(t_c-z_{a+1} )}\left(\frac {t_c-z_{a+1}-\ka k}{t_c-z_{a}}
+ \frac {{t_b-z_{a+1}-\ka k}}{t_b-z_{a}}
\,\frac{t_b-t_c+1-\ka k'}{t_b-t_c+1}\right)\\
&&=(w_I^\tau+w_I^{(a,a+1)\tau})(t,z_1,\dots,z_{a+1},z_a,\dots,z_n).
\eea
Here the common factor $h(z,t)$ does not depend on $z_a,z_{a+1}$. The case $b>c$ is similar. Therefore, the term $  {g}=w_Iv_I$ satisfies \eqref{eq}.

Third, let $I$ be such that $a\in I$ and $a+1\not\in I$. We pair it up with the set $J$ such that $a\not\in J$ and $a+1\in I$ in the most natural way. Namely, let $J=(I\setminus a )\cup\{a+1\}$ be the set $I$ where $a$ is replaced with $a+1$. Let $\tau\in S_l$ and $\tau(a)=b$. We claim that modulo $N$,
\bea
&&P R(z_1-z_{a+1}) (w_I^\tau(t,z)\, v_2\otimes v_1+ w_J^\tau(t,z)\, v_1\otimes v_2) \\
&&\hspace{30pt}\equiv w_I^\tau(t,z_1,\dots,z_{a+1},z_a,\dots,z_n)\, v_2\otimes v_1+ w_J^\tau(t,z_1,\dots,z_{a+1},z_a,\dots,z_n) \,v_1\otimes v_2.
\eea
Indeed, this is equivalent to 
\bea
w_I^\tau(t,z_1,\dots,z_{a+1},z_a,\dots,z_n)\equiv -\frac{1}{z_a-z_{a+1}-1}w_I^\tau(t,z)+\frac{z_a-z_{a+1}}{z_a-z_{a+1}-1} w_J^\tau, \\
w_j^\tau(t,z_1,\dots,z_{a+1},z_a,\dots,z_n)\equiv \frac{z_a-z_{a+1}}{z_a-z_{a+1}-1}w_I^\tau(t,z) -\frac{1}{z_a-z_{a+1}-1} w_J^\tau.
\eea 
For the first equation we have
\bea
&&-\frac{1}{z_a-z_{a+1}-1}w_I^\tau(t,z)+\frac{z_a-z_{a+1}}{z_a-z_{a+1}-1} w_J^\tau \\
&&= h(t,z)\left(\frac{1}{z_a-z_{a+1}-1}\frac{1}{t_b-z_a}-\frac{z_a-z_{a+1}}{z_a-z_{a+1}-1}\frac{t_b-z_a-\ka k'}{t_b-z_a}\frac{1}{t_b-z_{a+1}}\right)\\
&& \equiv \frac{h(t,z)}{t_b-z_{a+1}}=w_I^\tau(t,z_1,\dots,z_{a+1},z_a,\dots,z_n).
\eea
The second equation is similar. Therefore, the sum $  {g}=w_Iv_I+w_Jv_J$ satisfies \eqref{eq}.
Thus $g=w(t,z)$ satisfies  \eqref{eq}.
Equation \eqref{eq} is proved.
 
 \smallskip 
 
We note that if one replaces $k\kappa'$ in $w(t,z)$ by $2$, and $\kappa k$ by $-1$ then $  {w}(t,z)$ will satisfy \eqref{eq} exactly (not only modulo $N$). This property of the weight function $   w(t,z)$ is known in large generality and we could derive it, for example, from \cite[Theorem 4.9]{TV} instead of checking it directly.

Also note that in the complex-valued case, there is an extra ingredient, as one multiplies by a periodic function $W$ before integration, see Section \ref{app 2 sec}. This function is usually not symmetric in $z$ and therefore, the integrand loses symmetry \eqref{eq}.

\subsection{The $n$-th equation}
Clearly, for any $I\in\mc I_l$, 
\bea
(P^{(\mu)})^{-1} v_I=v_{\mu^{-1}I},  \qquad \mu^{-1}I=\{\mu^{-1}(a),\ a\in\ I\}.
\eea
To check the $n$-th equation 
\bea
(P^{(\mu)})^{-1} f_r(z_1-\kappa,z_2,\dots,z_n)&=& f_r(z_2,\dots,z_n,z_1), 
 \eea
we show that if one takes $   U_I^\tau(t,z_1-\kappa,z_2,\dots,z_n)$ and changes $t_i\to t_i-\kappa$, where $i$ depends on $I$ and $\tau$, then the result is 
$   U_{\mu^{-1}I}^{\tau'}(t,z_2,\dots,z_n,z_1)$, where
\bea 
\tau'=\begin{cases} \tau & 1\not\in I, \\  \tau\, (1,2,\dots,l) & 1 \in I.
\end{cases}
\eea
 
For that let us follow what happens to the strings when we shift $z_1$. The change $z_1 \to z_1-\ka$  affects only strings of the form $[t_i-z_1-\ka]_k$, and $[t_i-z_1-\ka]_{k-1}$ which are shifted to $[t_i-z_1]_k$, and  $[t_i-z_1]_{k-1}$ respectively. On Figure \ref{weight pic} it corresponds to the shift of the corresponding strings to the left by one position.

Let $I$ be such that  $1\not \in I$. Then the strings depending on $z_1$ in $   U_I^\tau(t,z)$ are only $[t_i-z_1-\ka]_{k}$, $i=1,\dots,l$. Note that $n\not \in I'=\mu^{-1}I$. Thus the strings depending on $z_1$
$   U_{\mu^{-1}I}^{\tau}(t,z_2,\dots,z_n,z_1)$ are $[t_i-z_1]_{k}$, $i=1,\dots,l$. Therefore,
 the shift $z_1\to z_1\kappa$ takes the strings in $   U_I^\tau(t,z)$ with the strings in $   U_{\mu^{-1}I}^{\tau}(t,z_2,\dots,z_n,z_1)$. The strings which do not involve $z_1$ match as well.
Thus we have $   U_I^\tau(t,z_1-\kappa,z_2,\dots,z_n)=   U_{\mu^{-1}I}^{\tau}(t,z_2,\dots,z_n,z_1)$.
Note that we did not shift $t_i$ in this case.

Let $I$ be such that  $1\in I$. Let $\tau(1)=b$. 
Then the strings depending on $z_1$ in $   U_I^\tau(t,z)$ are $[t_i-z_1-\ka]_{k}$, $i=1,\dots,l$, $i\neq b$ and $[t_b-z_1-\ka]_{k-1}$. Note that $\tau'(n)=\tau \mu(n)=\tau(1)=b$. Thus the strings depending on $z_1$
$   U_{\mu^{-1}I}^{\tau}(t,z_2,\dots,z_n,z_1)$ are $[t_i-z_1]_{k}$, $i\neq b$ and  $[t_b-z_1-\ka]_{k-1}$. 

We shift both $z_1\to z_1-\ka$ and $t_b\to t_b-\ka$ in 
$   U_I^\tau(t,z)$. Clearly, all strings are shifted appropriately to match the strings in $   U_{\mu^{-1}I}^{\tau'}(t,z_2,\dots,z_n,z_1)$.
Thus, in this case we have $   U_I^\tau(t_1,\dots,t_b-\ka,\dots,t_l,z_1-\kappa,z_2,\dots,z_n)=   U_{\mu^{-1}I}^{\tau}(t,z_2,\dots,z_n,z_1)$.

\smallskip 

Therefore, the $n$-th equation is valid for the integrand up to terms of the form $D_{t_b}U_I^\tau(t,z_1-\kappa,z_2,\dots,z_n)$ which vanish after taking the difference $r$-integral.

\subsection{$N$-hypergeometric solutions are singular vectors modulo $N$}
Finally, we claim that  the vector $   U(t,z)$ is a singular vector modulo $N$ up to difference $t$-differentials, namely,
\bea
%\label{sing}
e    U(t,z)\,\equiv\, \sum_{J\in \mc I_{l-1}} g_J(t,z) v_J \pmod{N},
\eea
where $g_J(t,z)$ are suitable discrete $t$-differentials.
This statement is deduced directly from \cite[Lemma 2.21]{TV}.  
Hence taking a difference integral of $eU(t,z)$ we obtain
the zero vector modulo $N$.

\begin{ex} Let $l=1$. Then 
$e    U(t,z) = \sum_{a=1}^n U_{\{a\}}(t_1,z)\, v_1\ox\dots \ox v_1\,,$
and modulo $N$ we have
\bea
&&\Phi(t_1-\ka,z) -\Phi(t_1, z)=\Phi(t,z) \Big[\prod_{a=1}^n \frac {t_1-z_{a}-\ka k}{t_1-z_{l}} - 1\Big]
\\
&&=
 -\Phi(t,z)\ka k \sum_{a=1}^n
\frac {1}{t_1-z_{a}} \prod_{l=1}^{a-1}\frac {t_1-z_{l}-\ka k}{t_1-z_{l}}
= -\ka k
\sum_{a=1}^n U_{\{a\}}(t_1,z)\equiv 
\sum_{a=1}^n U_{\{a\}}(t_1,z).
\eea
\end{ex}

Note, that one can modify solutions modulo $N$ by adding a vector valued polynomial $f(z)$ of the form \eqref{sol gen form} such that the coordinate polynomials $f_I(z)\in\Z[z]$ are all divisible by $N$. In many cases such a modification allows us to obtain solutions which are singular vectors (not only singular vectors modulo $N$), see Sections \ref{example sec}.

\subsection{Example.}\label{example sec}
Let $n=3$, $l=1$, $\ka=2$, $N=3$, $r=2$. Then $k=1$ and $   U(t_1,z_1,z_2,z_3)$ is a quadratic polynomial in $t_1$.
We have
\bea
\{   U(t_1,z_1,z_2,z_3)\}^{t_1}_2=v_2\otimes v_1\otimes v_1+
v_1\otimes v_2\otimes v_1+ v_1\otimes v_1\otimes v_2.
\eea
Note that the solution is a singular vector only modulo $3$. One can subtract, for example, $3 v_2\otimes v_1\otimes v_1$
to make it a singular vector.

\smallskip 

Let $n=3$, $l=1$, $\ka=2,$ $N=5$, $r=4$. Then $k=2$ and  $   U(t_1,z_1,z_2,z_3)$ is a polynomial of degree  $5$ in $t_1$.
We have
\bea
&&\{   U(t_1,z_1,z_2,z_3)\}^{t_1}_4=(14-z_1-2z_2-2z_3)v_2\otimes v_1\otimes v_1\\&&\hspace{70pt} +(10-z_2-2z_1-2z_3)
v_1\otimes v_2\otimes v_1
+ (6-z_3-2z_1-2z_2)v_1\otimes v_1\otimes v_2.
\eea
Again, the solution is a singular vector only modulo $5$ and one can subtract a multiple of $5$,
 for example, $(30-5(z_1+z_2+z_3)) v_2\otimes v_1\otimes v_1$ to make the solution a singular vector.

\smallskip 

More generally, for $n=3$, $l=1$, $\ka=2$, $N=2k+1$, the vector  $   U(t_1,z_1,z_2,z_3)$ 
is a polynomial of degree $3k-1$ in $t_1$ and we have a difference $r$-integral with $r=2k$. 
The corresponding solution modulo $N$ is a non-homogenous polynomial in $z_1,\dots,z_n$ of degree $k-1$.
 If $N=p$ is prime, then this is the only $N$-hypergeometric solution.

\section{From difference to differential equations}
\label{limit sec}
In this section we show that the solutions modulo $N$ constructed in Theorem \ref{main thm} recover the solutions of the KZ differential equations modulo $N$ constructed in \cite{SV}. For simplicity and following \cite{SV},  we restrict ourselves only to maximal difference $r$-integrals, that is to the case of $r$, such that $r_i\equiv -1 \pmod N$.

\subsection{$N$-hypergeometric  solutions of KZ differential equations}

In this section we remind the construction in \cite{SV} of polynomial solutions modulo an integer $N$ of the 
$\slt$ differential \KZ/ equations.

\vsk.2>
We consider the system of differential equations on a $V^{\ox n}$-valued function $f(z)$,
\bean
\label{KZ}
\ka \,\frac{\der f}{\der z_a} = \sum_{s,\ s\ne a} \frac{P^{(a,s)}-1}{z_a-z_s}\,f, \qquad a=1,\dots,n,
\eean
where $z=(z_1,\dots,z_n)$ and $\ka$ is a parameter of the system. This system is called the {\it \KZ/ differential equations}.
This system commutes with the $\slt$-action on $V^{\ox n}$.

\vsk.2>

Let  $0<\ka<N$ be the same relatively prime integers as before. Let $l$ be a positive integer, $2l\leq n$.
Let $k, k'$ be the same positive integers as in  \eqref{C}.  Let $t=(t_1,\dots,t_l)$.

\vsk.2>
Define the master polynomial 
\bea
\Phi^0(t,z) = \prod_{a=1}^n\prod_{i=1}^l (t_i-z_a)^k\prod_{1\leq i<j\leq l} (t_i-t_j)^{k'}\,.
\eea
For a function $f(t)$ define
\bea
\text{sym}_t \,f(t) &=&
\sum_{\tau\in S_l} \,f(t_{\tau(1)} ,\dots, t_{\tau(l)}) \,.
\eea
For $I=\{1\leq a_1<\dots<a_l\leq n\}\in \mc I_l$,
define the weight functions
\bea
w_I^0(t,z) = \text{sym}_t\,
\prod_{i=1}^l \frac {1}{t_i-z_{a_i}} \,,\qquad U_I^0(t,z)   =  \Phi^0(t,z) \,w_I^0(t,z).
\eea
Then $U_I^0(t,z)$ is a polynomial with integer coefficients. Expand $U^0_{I}(t,z)$ with respect to the $t$-variables
\bea
U_{I}^0(t,z) = \sum_{m_1,\dots,m_l\geq 0} c^0_{I, m}(z)\,t_1^{m_1}\dots t_l^{m_l}
\,,
\qquad c^0_{I, m}(z)\in\Z[z].
\eea
For any $r=(r_1,\dots,r_l)\in \Z_{>0}^l$  such that $r_i\equiv -1 \pmod N$, for all $i$, denote
\bean
\label{sol}
  {f}^{0}_r (z)  = \sum_{I\in\mc I_l} c^0_{I, r}(z)\,v_I \,.
\eean

\begin{thm}\cite{SV}
\label{SV2 thm} For any $r=(r_1,\dots,r_l)\in \Z_{>0}^l$  such that $r_i\equiv -1 \pmod N$, for $i=1,\dots,l$,
the polynomial $  {f}^{0}_r (z)$  satisfies the \KZ/ differential equations modulo $N$
and is a singular vector modulo $N$,
\bean
\label{sF0}
e   {f}_r^{0}(z) &\equiv& 0 
\phantom{aaaaaaaaaaaaa}   \pmod{N},
\\
\label{qF0}
\ka \,\frac{\partial   {f}_r^{0}}{\der z_a}
\,&\equiv&
\, \sum_{s, \ s\neq a} \frac{P^{(a,s)}-1}{z_a-z_s} \,  {f}^{0}_r \pmod{N},\qquad a=1,\dots,n.
\vv.3>
\eean

\end{thm}
The solutions $  {f}^{0}_r$ of the \KZ/ equations modulo $N$ given by Theorem \ref{SV2 thm} are called
the  $N$-hypergeometric solutions of the \KZ/ equations.

\smallskip  

Note that the polynomials $U_I^0(t,z)$ are homogeneous polynomials of variables $t,z$ of degree $nlk+l(l-1)k'/2-l$.
Hence the $N$-hypergeometric solution 
$  {f}_r^{0}(z)$ is a homogeneous polynomial in $z$ of degree
\bea
d_r=nlk+l(l-1)k'/2-l-\sum_{i-1}^l r_i.
\eea

\subsection{The top degree of the $N$-hypergeometric solutions}
Unlike the solutions $  {f}_r^{0}(z)$ of the \KZ/ equations, the solutions $  {f}_r(z)$  of the \qKZ/ equations given by Theorem \ref{main thm} are not homogeneous. It turns out that the top degree part
of these solutions coincides with $  {f}_r^{0}(z)$.

Define degree of polynomials in $\Z[t,z]$ by setting
\bea
\deg t_i=\deg z_j=1, \qquad i=1,\dots,l,\  j=1,\dots,n.
\eea

Let $r=(r_1,\dots,r_l)\in \Z_{>0}^l$  be such that $r_i\equiv -1 \pmod N$, for all $i$. 
Let ${f}_r(z)$ be the $N$-hypergeometric  solution of the \qKZ/ equations defined
 in Theorem \ref{main thm}.
Let ${f}_r^0(z)$ be the $N$-hypergeometric  solution of the \KZ/ equations defined
 in Theorem \ref{SV2 thm}.

\begin{prop}\label{qkz-kz}
We have
\bea
  {f}_r(z)=  {f}^{0}_r(z)+\dots,
\eea
where the dots denote the terms of degree less than $\deg   {f}^{0}_r(z)$.
\end{prop}
\begin{proof}
For a polynomial $X$ in variables $t,z$ and integer coefficients, we write $X=Y+\dots$ when $Y$ is a homogeneous polynomial in $t,z$ and all terms of $X-Y$ have degree smaller than the degree of $Y$.

For any integer $a$, we have
$t_i-t_j-a\ka]_m=(t_i-t_j)^m+\dots$ and $[t_i-z_j-a\ka]_m=(t_i-z_j)^m+\dots$.
Hence  $U(t,z)=   U^0(t,z)+\dots$.
At the same time if 
\bea
   U(t,z)=\sum_m   {c}_m(z) \prod_{i=1}^l [t_i]_{m_i}, \qquad   {c}_m(z)=  {c}_m^{0}(z) +\dots,
\eea 
then $U(t,z)=\sum_m   {c}_m^{0}(z) \prod_{i=1}^lt_i^{m_i}+\dots$.
Therefore, 
$U^0(t,z)=\sum_m   {c}_m^{0}(z) \prod_{i=1}^lt_i^{m_i}$\,,
\ and the proposition is proved.
\end{proof}

Proposition \ref{qkz-kz} says that there are at least as many $N$-hypergeometric
 solutions of the \qKZ/ equations as the $N$-hypergeometric solutions of the \KZ/ equations.

\iffalse
\begin{cor} 
The rank of the module of $N$-hypergeometric solutions of the \qKZ/  is greater or equal to the rank of the module of  $N$-hypergeometric solutions of the \KZ/ equation. \qed
\end{cor}
\fi

It is an interesting problem to describe the module of the $N$-hypergeometric solutions in either case.
 The answer depends on arithmetic properties of $N,n,\ka,l$.

\appendix

\section{Hypergeometric solutions}
\label{sec a}
In this section we remind the limit which produces the \KZ/ equations from the \qKZ/ equations. 
We also remind the formulas for the hypergeometric complex-valued solutions 
of the $\slt$ \KZ/  differential and \qKZ/ difference equations
with values in $\Sing \,V^{\ox n}[n-2]$, that is, in the case when $l=1$.

\subsection{\KZ/ equations as a limit of \qKZ/ equations}\label{eqn limit sec}
It is well known, that the \KZ/ equations can be obtained as a limit of the
\qKZ/ equations which explains the origin of Proposition \ref{qkz-kz}. For the reader's convenience, we recall the construction.

Let $f(z_1,\dots, z_n)$ satisfy the \qKZ/ equations,
\beq
\label{Ke}
f(z_1,\dots, z_a-\ka,\dots, z_n)\,=\,
K_a(z\:;\ka)\,f(z),\qquad a=1,\dots,n.
\vv.3>
\eeq
Define  
$g(w_1,\dots,w_n;\al) = f(w_1/\al, \dots, w_n/\al)$.
Then
\beq
\label{K)}
g(w_1,\dots, w_a-\al\ka,\dots, w_n)\,=\,
K_a(w/\al\:;\ka)\,g(w),\qquad a=1,\dots,n,
\vv.3>
\eeq
where
\bean
\label{K, alpha}
K_a(w/\al;\ka) 
&=&
R^{(a,a-1)}(w_a-w_{a-1}-\al \ka;\al)\,\dots\, R^{(a,1)}(w_a-w_1-\al\ka;\al)
\\
\notag
&\times&
 R^{(a,n)}(w_a-w_n;\al)\,\dots\,R^{(a,a+1)}(w_a\<-w_{a+1};\al)\,,
\eean
and 
\bea
R(u;\al)=\frac{u-\al P}{u-\al} = 1\,-\, \al\,\frac{P-1}{u-\al} \,.
\eea
In the limit $\al\to 0$, the \qKZ/ difference equations \eqref{K, alpha} turn into the \KZ/ differential equations
\eqref{KZ}.

\subsection{Solutions of the \KZ/ equations}
\label{sec a1}

Define the master function and weight functions
\bea
\Phi(t_1,z) = \prod_{a=1}^n(t_1-z_a)^{-1/\ka},
\qquad
w_{\{a\}}(t_1,z) = \frac 1{t-z_a}\ , \qquad a=1,\dots, n.
\eea
Consider the vector of integrals
\bea
  {F}_{(\ga)}^0(z)\ =\ \sum_{a=1}^n \ \int_{\ga(z)} \Phi(t_1,z) \,w_{\{a\}}(t_1,z)\,dt_1 \, v_{\{a\}}\,  \in \ V^{\ox n}[n-2]\,,
\eea
where $\ga(z)$ is a flat family of 1-cycles associated with the multi-valued master function
$ \Phi(t_1,z)$.
Then $  {F}_{(\ga)}^0(z)$ is a solution of the \KZ/ differential equations \eqref{KZ} and $e  {F}_{(\ga)}^0(z)=0$.

\vsk.2>
In this case the solutions are labeled by the families of cycles $\ga(z)$.

\subsection{Solutions of the \qKZ/ equations}\label{app 2 sec}
Define the master function and weight functions
\bea
\Phi(t_1,z) = \prod_{a=1}^n\frac{\Ga\big(\frac {t_1-z_a+1}{-\ka}\big)}{\Ga\big(\frac {t_1-z_a}{-\ka}\big)},
\qquad
w_{\{a\}}(t_1,z) = \frac 1{t_1-z_a}\prod_{l=1}^{a-1} \frac{t_1-z_l+1}{t-z_l}\ , \qquad a=1,\dots, n.
\eea
Define the trigonometric weight functions 
\bea
W_{\{b\}} (t_1,z) \,=\, \frac{ e^{\pi i (t-z_b)/\ka}}{\sin(\pi (t-z_b)/\ka)}\,\prod_{l=1}^{b-1}
\,\frac{\sin(\pi (t-z_l+1)/\ka)}{\sin(\pi (t-z_l)/\ka)}\,, 
\eea
\bea
W_{\{b\}}^{\text{sing}} (t_1,z) = W_{\{b\}} (t_1,z)  -  W_{\{b+1\}} (t_1,z)\,,
\qquad b=1,\dots,n-1.
\eea
Consider the vector of integrals
\bea
  {F}_{\{b\}}(z)\ =\ \sum_{a=1}^n \ \int_{i\R}  
\Phi(t_1,z) \,w_{\{a\}}(t_1,z)\,W_{\{b\}}^{\text{sing}} (t_1,z) \,dt_1 \, v_{\{a\}}\,  \in \ V^{\ox n}[n-2]\,,  
\eea
$ b=1,\dots,n-1$, where $i\R$ is the imaginary axis properly deformed depending on values of $z$.
Then $  {F}_{\{b\}}(z)$ is a solution of the \qKZ/ equations \eqref{Ki} and $e  {F}_{\{b\}}(z)=0$, see \cite[Section 8]{TV}.

\vsk.2>
In this case the solutions are labeled by the trigonometric weight functions $W_{\{b\}}^{\text{sing}} (t_1,z) $.

\vsk.2>

One may study the functions $  {F}_{\{b\}}(z)$ under the limit $\al \to 0$ as in Section \ref{eqn limit sec} and show that the limit
gives solutions $  {F}_{(\ga)}^0(z)$ of Section \ref{sec a1}, see \cite[Section 8]{TV}. The proof of this fact is rather delicate, if
compared with the trivial observation in Proposition \ref{qkz-kz}.

\end{document}